\documentclass[a4paper,twoside,12pt]{article}
\usepackage{amssymb,amsmath,amsthm,latexsym}
\usepackage{amsfonts}
\usepackage{graphicx,caption,subcaption}
\usepackage[pdftex,bookmarks,colorlinks=false]{hyperref}
\usepackage[hmargin=1.2in,vmargin=1.2in]{geometry}
\usepackage[mathscr]{euscript}
\usepackage{float}
\usepackage[auth-lg,affil-it]{authblk}
\numberwithin{equation}{section}

\allowdisplaybreaks

\def\ni {\noindent}

\newcommand{\N}{\mathbb{N}}
\newcommand{\C}{\mathbb{C}}

\newtheorem{thm}{Theorem}[section]

\newtheorem{lem}[thm]{Lemma}
\newtheorem{prop}[thm]{Proposition}
\newtheorem{cor}[thm]{Corollary}

\newtheorem{rem}{Remark}[section]
\newtheorem{ill}{Illustration}[section]

\title{\textbf{Chromatic Zagreb Indices for Graphical Embodiment of Colour Clusters}}
 
\vspace{0.5cm}
\author{Johan Kok}
\affil{\small Tshwane Metropolitan Police Department, \\ City of Tshwane, South Africa\\ {\tt kokkiek2@tshwane.gov.za}}

\author{Sudev Naduvath}
\affil{\small Centre for Studies in Discrete Mathematics\\ Vidya Academy of Science \& Technology\\ Thrissur, India.\\ {\tt sudevnk@gmail.com}}

\author{Muhammad Kamran Jamil}
\affil{\small Department of Mathematics\\ Riphah Institute of Computing and Applied Sciences\\ Riphah International University\\ Lahore, Pakistan.\\ {\tt m.kamran.sms@gmail.com}}

\date{}

\begin{document}
\maketitle 
	
\begin{abstract}
\ni  For a colour cluster $\C =(\mathcal{C}_1,\mathcal{C}_2, \mathcal{C}_3,\ldots,\mathcal{C}_\ell)$, where $\mathcal{C}_i$ is a colour class such that $|\mathcal{C}_i|=r_i$, a positive integer, we investigate two types of simple connected graph structures $G^{\C}_1$, $G^{\C}_2$ which represent graphical embodiments of the colour cluster such that the chromatic numbers $\chi(G^{\C}_1)=\chi(G^{\C}_2)=\ell$ and $\min\{\varepsilon(G^{\C}_1)\}=\min\{\varepsilon(G^{\C}_2)\} =\sum\limits_{i=1}^{\ell}r_i-1$. Therefore, the problem is the edge-minimality inverse to finding the chromatic number of a given simple connected graph. In this paper, we also discuss the chromatic Zagreb indices corresponding to $G^{\C}_1$, $G^{\C}_2$. 
\end{abstract}

\ni \textbf{Keywords:} Graphical embodiments; colour clusters; colour classes; chromatic Zagreb indices.

\vspace{0.25cm}

\ni \textbf{Mathematics Subject Classification:} 05C15, 05C38, 05C75, 05C85. 

\section{Introduction}
For general notation and concepts in graphs and digraphs see \cite{BM1,CL1,FH,DBW}. Unless mentioned otherwise all graphs we consider in this paper are finite, simple, connected and undirected graphs.
 
The \textit{chromatic number} $\chi(G)\geq 1$ of a graph $G$ is the minimum number of distinct colours that allow a proper colouring of $G$. Such a colouring is called a \textit{chromatic colouring}. of $G$.  The colour weight of a colour in a graph $G$ is the number of times a colour $c_j$ has been allocated to the vertices of $G$, denoted by $\theta(c_j) \geq 1$. In other words, the colour weight of a colour $c_i$ can be considered as the cardinality of the corresponding colour class in the graph concerned.  

Analogous to set theory notation, let $\C = \bigcup\limits_{1\leq i \leq \chi(G)}\mathcal{C}_i$, and be called a colour cluster. It is noted that for $\chi(G)\geq 2$ a graph $G^{\C}$ with $\max\{\varepsilon(G^{\C})\}$ that requires the colour cluster $\C$ to allow exactly a chromatic colouring (minimum proper colouring) is the complete $\chi(G)$-partite graph $K_{\theta(c_1),\theta(c_2),\ldots,\theta(c_{\chi(G)})}$ or put differently, the complete $\ell$-partite graph with vertex partitioning $V_i(G^{\C})$, $1\leq i \leq \ell$ such that $|V_i(G^{\C})| = |\mathcal{C}_i|$, and $\bigcup\limits_{1\leq i \leq \ell} V_i(G^{\C}) = V(G^{\C})$. Therefore, if $|\mathcal{C}_i|=1$, $\forall~ 1\leq i\leq \chi(G)$ it is a complete graph, $K_{\chi(G)}$. For the colour cluster $\C=(\mathcal{C}_1)$ the graph $G^{\C}$ is a null graph (no edges).

Energy can entirely be transformed and transmitted as energy waves. The electromagnetic spectrum of these energy waves ranges between cosmic rays, gamma rays, X-rays, ultra violet rays, optical light with its inherent spectrum of optical colours, infra-red rays, micro waves, short radio waves and long radio waves. These waves or rays transmit at the speed of light. An energy unit per wave (or ray) type is the amount of energy transmitting per second. All such energy units are simply called, colours. If for a given finite cluster of distinct energy units (colours), the colours must be linked to ensure connectivity but same colours must remain apart to prevent merging, then the problem we investigate is that of finding a simple connected graphical embodiment allowing the colours as a proper colouring. More specifically the problem to be considered is to find properties and related results of a connected graph $G^{\C}$ with $\min\{\varepsilon(G^{\C})\}$ that allows the prescribed colour cluster $\C$ as a chromatic colouring. Henceforth the graph $G^{\C}$ will mean a graph with minimum edges allowing $\C$ as a chromatic colouring.

\section{Graphical Embodiment of $\C$}

We note that the cycle $C_5$ has $\chi(C_5)=3$ with colour cluster say, $\C =(c_1,c_1,c_2,c_2,c_3)$. However, path $P_5$ is a graph with minimum number of edges that allows $\C$ as a proper colouring (not chromatic colouring). This observation leads us to our first useful result given below.

\begin{lem}\label{Lem-2.1}
For any colour cluster $\C = (\mathcal{C}_1,\mathcal{C}_2,\mathcal{C}_3,\ldots,\mathcal{C}_\ell)$ with $|\mathcal{C}_i| = r_i >0,\ \ell\geq 2,\ 1\leq i \leq \ell$, the connected graphical embodiment with minimum edges that allows $\C$ a proper colouring has $\sum\limits_{i=1}^{\ell}r_i -1$ edges. 
\end{lem}
\begin{proof}
Label the $r_i$ vertices corresponding to $\mathcal{C}_i$ to be $v_{i,1},v_{i,2},v_{i,3},\ldots,v_{i,r_i}$, $1\leq i \leq \ell$. Add the edges $v_{1,1}v_{j,k}$, for all $2\leq j \leq \ell,\ 1\leq k \leq r_j$. Also, add the edges $v_{1,i}v_{2,1}$, $2\leq i\leq r_1$. Clearly, the graph $G_1$ obtained is a tree for which $\C$ is a proper colouring. Also, since $G_1$ is a connected acyclic graph, it has minimum number of edges, $\sum\limits_{i=1}^{\ell}r_i -1$.
\end{proof}

The construction methodology mentioned in the proof of Lemma \ref{Lem-2.1} may be called the \textit{Type-I graphical embodiment}. In order to obtain a $G^{\C}$, we consider the induced star subgraph $\langle v_{1,1},v_{2,1},v_{3,1},\ldots,v_{\ell,1}\rangle$ of $G_1$ and add the edges to obtain a complete subgraph, $K_\ell$. Note that $\frac{1}{2}(\ell -1)(\ell -2)$ edges were added to the graph $G_1$.

An alternative construction of such a graph $G^{\C}$ can be done by utilising the initial edges $v_{i,1}v_{i+1,j}$, $1\leq i \leq \ell -1$, $1\leq j\leq r_{i+1}$ as well as the edges $v_{2,1}v_{1,j}$, $2\leq j \leq r_1$ to obtain graph $G_2$. The aforesaid construction is called the \textit{Type-II graphical embodiment}. Thereafter, consider the induced path subgraph $\langle v_{1,1},v_{2,1},v_{3,1},\ldots,v_{\ell,1}\rangle$ and add the required $\frac{1}{2}(\ell -1)(\ell -2)$ edges to obtain the complete subgraph, $K_{\ell}$. This brings us to the existence of the following theorem.

\begin{thm}
For any colour cluster $\C$, at least one graphical embodiment exist with minimum number of edges for which $\C$ is a chromatic colouring and $\chi(G^{\C})=\ell$.
\end{thm}
\begin{proof}
As stated above, for the colour cluster $\C=(\mathcal{C}_i)\ 1\leq i \leq \ell$ and $|\mathcal{C}_i|=|\mathcal{C}_j|=1$, for all values of $i,j$, the graph $G^{\C}$ is complete. Hence, at least one graph with required properties exists.

For all other colour clusters, the result is a direct consequence of Lemma \ref{Lem-2.1} and the alternative construction of $G^{\C}$.
\end{proof}

We note that for a colour cluster $\C =(\mathcal{C}_1, \mathcal{C}_2)$ the graph $G^{\C}$ is acyclic. For colour clusters $\C = (\mathcal{C}_i)$, $i \geq 3$ the graphical embodiment $G^{\C}$ has a unique maximum clique. Figure 1(a) and Figure 1(b) depict $G^{\C}_1$, $G^{\C}_2$ for $\C = (c_1,c_1,c_1,c_1,c_1,c_2,c_2,c_2,c_2,c_3,c_3,c_3,\\c_4,c_4,c_4)= ((c_1,c_1,c_1,c_1,c_1),(c_2,c_2,c_2,c_2),(c_3,c_3,c_3),(c_4,c_4,c_4))$.

\begin{figure*}[h!]
	\centering
	\begin{subfigure}[b]{0.45\textwidth}
		\centering
		\includegraphics[height=2in]{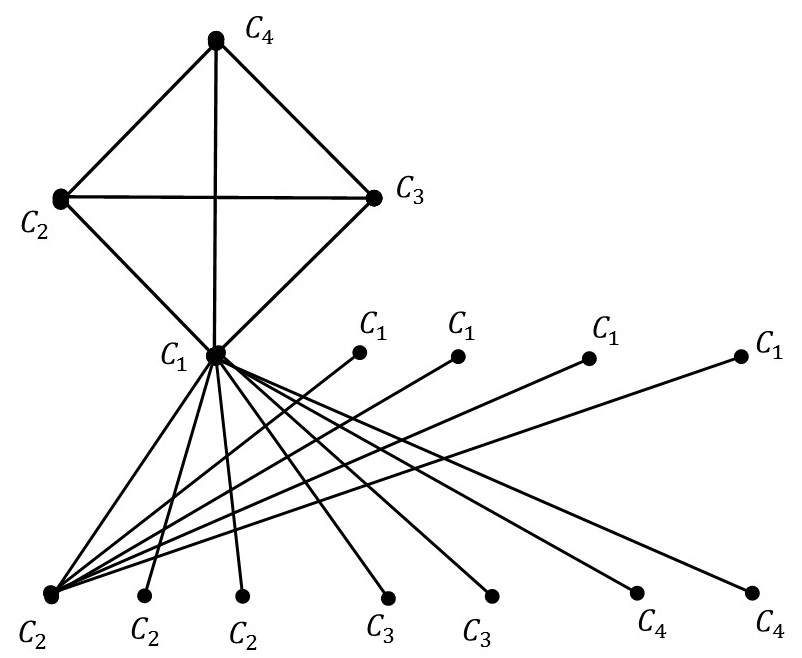}
		\caption{}\label{fig:fig-1a}
	\end{subfigure}%
	\quad 
	\begin{subfigure}[b]{0.45\textwidth}
		\centering
		\includegraphics[height=2in]{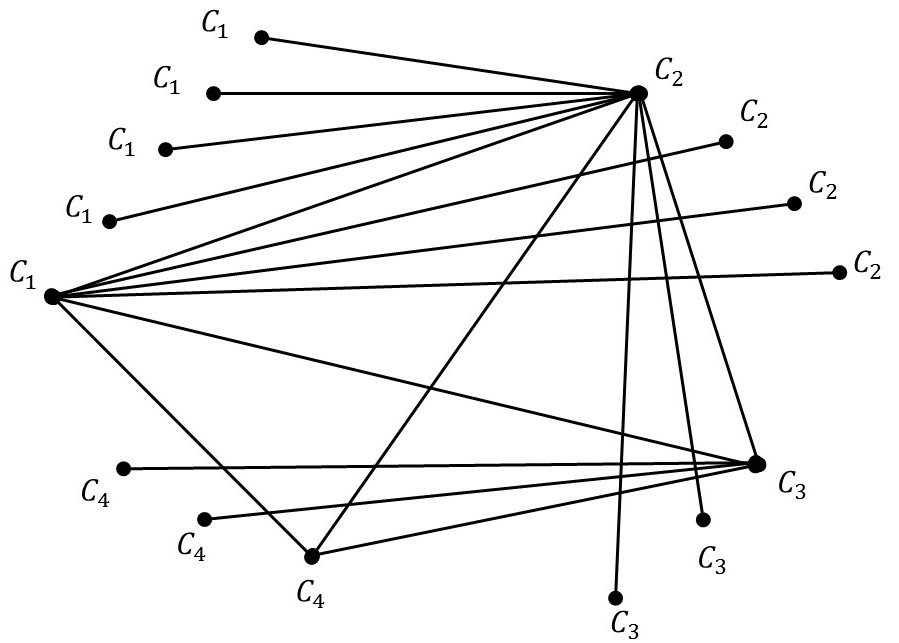}
		\caption{}\label{fig:fig-1b}
	\end{subfigure}
	\caption{}\label{fig:fig-1}
\end{figure*}

\section{Chromatic Zagreb Indices of $G^{\C}$}

The topological graph indices related to irregularity of a graph namely the \textit{first Zagreb index}, $M_1(G)$ and the \textit{second Zagreb index}, $M_2(G)$ are of the oldest irregularity measures researched. A new \textit{irregularity} of $G$ has been introduced in \cite{MOA} as $irr(G)=\sum \limits _{e\in E(G)} imb(e)$, where $imb(e)=|d(v)-d(u)|$. In \cite{GHF}, this irregularity index was named the \textit{third Zagreb index} to conform with the terminology of chemical graph theory. Recently, the topological index called \textit{total irregularity} has been introduced in \cite{ABD1} as $irr_t(G) =\frac{1}{2}\sum \limits_{u,v \in V(G)}|d(u) - d(v)|$. This irregularity index may be called the \textit{fourth Zagreb index}. The study of Zagreb indices is strongly dependent on the structural property of the edges and subsequently the degree of the vertices of the graph under study.

A vertex colouring with colours of minimum subscripts is called a \textit{minimum parameter} colouring. Unless stated otherwise, the colour sets we consider in this paper are minimum parameter colour sets. The weight of a colour $c_i$ is defined to be the number of times the colour $c_i$ is allocated to vertices of a graph $G$ and is denoted by $\theta(c_i)$.   If $\varphi:v_i \mapsto c_j$, then we write $c(v_i) = c_j$.

The notion of chromatic Zagreb indices was introduced in \cite{KSM}. For ease of reference we also recall that the first three Zagreb indices are defined as:

\begin{eqnarray}
M_1(G) & = & \sum\limits_{i=1}^{n}d^2(v_i) = \sum \limits_{i=1}^{n-1}\sum\limits_{j=2}^{n} (d(v_i) + d(v_j))_{v_iv_j \in E(G)}\\
M_2(G) & = & \sum \limits_{i=1}^{n-1}\sum\limits_{j=2}^{n} d(v_i)d(v_j)_{v_iv_j \in E(G)}\\
M_3(G) & = & \sum \limits_{i=1}^{n-1}\sum\limits_{j=2}^{n} |d(v_i) - d(v_j)|_{v_iv_j \in E(G)}
\end{eqnarray}

We define the default values, $M_1(K_1)= M_2(K_1)=M_3(K_1)=0$. Note that the Zagreb indices are all functions of the degree of vertices of graph $G$. For a given graph, the vertex degrees are invariants. If the  invariants $d(v_i),\ \forall\, v_i \in V(G)$ are replaced by the parameters $s$, $c(v_i) = c_s$, $\forall v_i \in V(G)$ the \textit{chromatic Zagreb indices} are defined. Note that for any minimum parameter set of colours $\mathcal{C}$, $|\mathcal{C}| = \ell$, a graph $G$ has $\ell!$ minimum parameter colourings. Denote these colourings $\varphi_t(G)$, $1\leq t \leq \ell!$. Define the variable chromatic Zagreb indices as:
\begin{eqnarray}
M^{\varphi_t}_1(G) & = & \sum\limits_{i=1}^{n}s^2, c(v_i)=c_s,1\leq t \leq \ell!,\\ 
 & = & \sum\limits_{j=1}^{\ell}\theta(c_j)\cdot j^2,\ c_j \in \mathcal{C},\\
M^{\varphi_t}_2(G) & = & \sum \limits_{i=1}^{n-1}\sum\limits_{j=2}^{n} (s\cdot k)_{v_iv_j \in E(G)},\ c(v_i)= s,\ c(v_j)=k,\ 1\leq t \leq \ell!,\\
M^{\varphi_t}_3(G) & = & \sum \limits_{i=1}^{n-1}\sum\limits_{j=2}^{n} |s-k|_{v_iv_j \in E(G)},\  c(v_i)= s,\ c(v_j)=k,\ 1\leq t \leq \ell!.
\end{eqnarray}
 
From the above we define the minimum and maximum chromatic Zagreb indices as follows:

\begin{eqnarray*}
M^{\varphi^-}_1(G) & = & \min\{M^{\varphi_t}_1(G): 1\leq t \leq \ell!\}, \\
M^{\varphi^+}_1(G) & = & \max\{M^{\varphi_t}_1(G): 1\leq t \leq \ell!\}, \\
M^{\varphi^-}_2(G) & = & \min\{M^{\varphi_t}_2(G): 1\leq t \leq \ell!\},\\
M^{\varphi^+}_2(G) & = & \max\{M^{\varphi_t}_2(G): 1\leq t \leq \ell!\},\\
M^{\varphi^-}_3(G) & = & \min\{M^{\varphi_t}_3(G): 1\leq t \leq \ell!\},\\
M^{\varphi^+}_3(G) & = & \max\{M^{\varphi_t}_3(G):1\leq t \leq \ell!\}.
\end{eqnarray*}

By convention we define the default values $M^{\varphi^-}_2(K_1)=M^{\varphi^+}_2(K_1)=0$, and $M^{\varphi^-}_3(K_1) = M^{\varphi^+}_3(K_1)= 1$.

\ni In view of the above mentioned terminology, we also observe the following useful lemma.

\begin{lem}\label{Lem-3.1}
For a graph $G$ with chromatic colouring $\mathcal{C} = (c_1,c_2,c_3,\ldots,c_k)$, the graph $G+e$, $e =c_ic_j,\ i\neq j$ has $M^{\varphi_t}_1(G)=M^{\varphi_t}_1(G+e)$ and therefore, $M^{\varphi^-}_1(G)=M^{\varphi^-}_1(G+e)$ and $M^{\varphi^+}_1(G)=M^{\varphi^+}_1(G+e)$.
\end{lem}
\begin{proof}
Because the $M^{\varphi_t}_1(G)$ is defined in terms of the colour subscript only and $\chi(G) = \chi(G+e)$ the result follows immediately.
\end{proof}

\subsection{Application to $G^{\C}$}

In this discussion, we will consider a specific initiation colour cluster from amongst the $\ell!$ orderings of the colour classes. That is, we consider a colour cluster $\C = (\mathcal{C}_1,\mathcal{C}_2,\mathcal{C}_3,\ldots,\mathcal{C}_\ell)$ with $|\mathcal{C}_i|\geq |\mathcal{C}_{i+1}|$, $1 \leq i \leq \ell -1$. Note that $\theta(c_i)=|\mathcal{C}_i|=r_i$. It implies that if after a mapping, say $c_j\mapsto c_{f(i)}$, then $\theta(c_j) \mapsto \theta(c_{f(i)})$ and $r_j \mapsto r_{f(i)}$. Hence, the notation styles may be used interchangeably depending on which notation proffers the idea best. The convention is to determine the minimum chromatic Zagreb index directly in terms of the initiation colour cluster and the definition of construction. Determining the maximum chromatic Zagreb index will follow from an appropriate colour mapping.

\ni First, we recall important results that have been proved in \cite{KSM}.

\begin{thm}
{\rm \cite{KSM}} For a finite tree of order $n\geq 4$, we have
\begin{enumerate}\itemsep0mm
\item[(i)] $n+ 3 \leq M^{\varphi^-}_1(T) \leq M^{\varphi^+}_1(T)\leq 4n- 3$;
\item[(ii)] $M^{\varphi^-}_2(T) = M^{\varphi^+}_2(T) =2(n-1)$;
\item[(iii)] $M^{\varphi^-}_3(T) = M^{\varphi^+}_3(T) = n-1$.
\end{enumerate}
\end{thm}

Applying Lemma \ref{Lem-3.1} to the $G^{\C}$, where $\C=(\mathcal{C}_1,\mathcal{C}_2)$, we have the next result as follows.

\begin{cor}
\begin{enumerate}\itemsep0mm
\item[(i)] $\theta(c_1)+\theta(c_2)+3 \leq M^{\varphi^-}_1(G^{\C}) \leq M^{\varphi^+}_1(G^{\C})\leq 4(\theta(c_1) + \theta(c_2))- 3$.
\item[(ii)] $M^{\varphi^-}_2(G^{\C}) = M^{\varphi^+}_2(G^{\C}) =2(\theta(c_1) + \theta(c_2)-1)$.
\item[(iii)] $M^{\varphi^-}_3(G^{\C}) = M^{\varphi^+}_3(G^{\C}) = \theta(c_1) + \theta(c_2)-1$.
\end{enumerate}
\end{cor}

\begin{prop}\label{Prop-3.3}
{\rm \cite{KSM}}. For complete $r$-partite graphs, $r\geq 2$, we have
\begin{enumerate}\itemsep0mm 
\item[(i)] $M^{\varphi_t}_1(K_{\underbrace{n,n,n,\ldots,n}_{r-entries}}) = \frac{n}{6}r(r+1)(2r+1)$.
\item[(ii)] $M^{\varphi_t}_2(K_{\underbrace{n,n,n,\ldots,n}_{r-entries}}) =\frac{n^2}{2}\sum\limits_{i=2}^{r}i^2(i-1)$.
\item[(iii)] $M^{\varphi_t}_3(K_{\underbrace{n,n,n,\ldots,n}_{r-entries}}) = n^2\sum\limits_{i=1}^{r-1}i(r-1)$.
\end{enumerate}
\end{prop}

Proposition \ref{Prop-3.3} finds immediate application for $G^{\C}_{max~edges}$ if $|\mathcal{C}_i| = |\mathcal{C}_j|$, $\forall i,j$. We also present closed formula for Proposition \ref{Prop-3.3} (ii),(iii).

\begin{lem} For a complete $r$-partite graph $K_{\underbrace{n,n,n,\ldots,n}_{r-entries}}$, we have\\ 
(i) $M^{\varphi_t}_2(K_{\underbrace{n,n,n,\ldots,n}_{r-entries}}) = \frac{n^2[2r^4 - r(2r+1)(r+3)]}{24}$.\\
(ii) $M^{\varphi_t}_3(K_{\underbrace{n,n,n,\ldots,n}_{r-entries}}) = \frac{n^2r(r-1)^2}{2}$.
\end{lem}
\begin{proof} 
We proof the results through substitution.

(i) Since $\frac{n^2}{2}\sum\limits_{i=2}^{r}i^2(i-1) = \frac{n^2}{2}[\sum\limits_{i=2}^{r}i^3 - \sum\limits_{i=2}^{r}i^2] = \frac{n^2}{2}[\sum\limits_{i=2}^{r}i^3 - \sum\limits_{i=2}^{r}i^2] = \frac{n^2}{2}[\sum\limits_{i=1}^{r}i^3 - \sum\limits_{i=1}^{r}i^2]$ and 
$\sum\limits_{i=1}^{r}i^3 = (\frac{r(r+1)}{2})^2$ and $\sum\limits_{i=1}^{r}i^2 = \frac{r(r+1)(2r+1)}{6}$, the result follows through substitution and simplification.

(ii) Since $n^2\sum\limits_{i=1}^{r-1}i(r-1) = n^2(r-1)\sum\limits_{i=1}^{r-1}i$ and $\sum\limits_{i=1}^{r-1}i =\frac{r(r-1)}{2}$, the result follows.
\end{proof}

Note that both Type-I and II construction of a graph $G$ with minimum edges which allows $\C$ as a proper colouring have $\chi(G) \leq \ell$. This will be true for any other construction type. Therefore, each construction type is extended to a well defined  complete subgraph to obtain $G^{\C}$, except for the case $\ell = 2$.  From \cite{KSM} the chromatic Zagreb indices for complete graphs are known.

\begin{thm}
Any two graphs $G^{\C}_1$, $G^{\C}_2$ which allow the colour cluster $\C$ as a chromatic colouring have $M^{\C^-}_1(G^{\C}_1) = M^{\C^-}_1(G^{\C}_2)$ and $M^{\C^+}_1(G^{\C}_1) = M^{\C^+}_1(G^{\C}_2)$.
\end{thm}
\begin{proof}
The result is a direct consequence of the definition.
\end{proof}

The following theorem discusses whether a thorn complete graph can be the graphical embodiment of a colour cluster.

\begin{thm}\label{Thm-3.7}
For a colour cluster $\C = (\mathcal{C}_i)\ 1\leq i \leq \ell$, any thorn complete graph $K^\star_\ell$ of order $\sum\limits_{i=1}^{\ell}|\mathcal{C}_i|$ corresponds to a graphical embodiment of the colour cluster with $\min\{\varepsilon(K^\star_\ell)\}$.
\end{thm}
\begin{proof}
Since, $\chi(K^\star_\ell) = \ell$ and a minimum number of pendant vertices exist to allocate exactly $|\mathcal{C}_i| - 1$, $1\leq i \leq \ell$ colours to pendant vertices, the result follows immediately.
\end{proof}

Note that Theorem \ref{Thm-3.7} provides a sufficient condition only. We know that all odd cycles $C_n,\  n\geq 5$, have chromatic number $3$. These cycles are triangle-free and have $\min\{\varepsilon(C_n)\}$ for the corresponding colour cluster $(\underbrace{(c_1,c_1,\ldots,c_1)}_{t~entries},\underbrace{(c_2,c_2,\ldots,c_2)}_{t~entries},(c_3))$, where $n = 2t + 1$. 

We extend this study by analysing the chromatic Zagreb indices for the Type-I and II constructions for a colour cluster $\C = (\mathcal{C}_1,\mathcal{C}_2,\mathcal{C}_3,\ldots,\mathcal{C}_\ell)$ with $|\mathcal{C}_i|\geq |\mathcal{C}_{i+1}|$, $1 \leq i \leq \ell -1$.

\subsection{Analysis of chromatic Zagreb indices for $G^{\C}_1$}

\begin{prop}
For the colour cluster $\C$ the graph $G_1$ obtained from the Type-I construction has
\begin{eqnarray*}
(i)\ M^{\C^-}_1(G_1) & = & \sum\limits_{i=1}^{\ell}\theta(c_i)\cdot i^2;\\
    M^{\C^+}_1(G_1) & = & \sum\limits_{i=1}^{\ell}\theta(c_i)\cdot (\ell -(i-1))^2.\\
(ii)\ M^{\C^-}_2(G_1) & = & 2(\ell-1) + \sum\limits_{i=2}^{\ell}\theta(c_i)\cdot i;\\
M^{\C^+}_2(G_1) & = & \ell(\ell-1)^2 +\sum\limits_{i=2}^{\ell}\theta(c_i)\cdot \ell(\ell-(i-1)).\\
(iii)\ M^{\C^-}_3(G_1) & = & \ell-1 +\sum\limits_{i=2}^{\ell}\theta(c_i)\cdot (i-1);\\
M^{\C^+}_3(G_1) & = & (\ell -1)^2 +\sum\limits_{i=2}^{\ell}\theta(c_i)\cdot (\ell-(i-1)).
\end{eqnarray*}
\end{prop}
\begin{proof}
\begin{enumerate}
\item[(i)] Follows directly from the definition of the first chromatic Zagreb index.

\item[(ii)] Each edge $v_{2,1}v_{i,j}$, $2\leq j\leq r_1$ has product value of 2. Therefore the first term. The second term follows by similar reasoning in respect of each colour $c_i$, $2\leq i \leq \ell$. Therefore the second term hence, the result for the minimum second chromatic Zagreb index follows. The maximum second chromatic Zagreb index follows similarly in respect of the map $c_i \mapsto c_{\ell-(i-1)}$.

\item[(iii)] The minimum third chromatic Zagreb index follows by similar reasoning as in part (ii) and the maximum third chromatic Zagreb index follows in respect of the map $c_\ell\mapsto c_1, c_1\mapsto c_2,c_2\mapsto c_3 \ldots c_{\ell-2}\mapsto c_{\ell-1}$.
\end{enumerate}
\end{proof}

\begin{cor}
For the colour cluster $\C$, the graph $G^{\C}_1$ has:
\begin{eqnarray*}
(i)\ M^{\C^-}_1(G^{\C}_1) & = & \sum\limits_{i=1}^{\ell}\theta(c_i)\cdot i^2;\\
M^{\C^+}_1(G^{\C}_1) & = & \sum\limits_{i=1}^{\ell}\theta(c_i)\cdot (\ell -(i-1))^2.\\
(ii)\ M^{\C^-}_2(G^{\C}_1) & = & 2(\ell-1) + \sum\limits_{i=2}^{\ell}\theta(c_i)\cdot i +\sum\limits_{j=2}^{\ell-1}\sum\limits_{i=j+1}^{\ell}j\cdot i;\\
M^{\C^+}_2(G^{\C}_1) & = & \ell(\ell-1)^2 +\sum\limits_{i=2}^{\ell}\theta(c_i)\cdot \ell(\ell-(i-1))+\sum\limits_{j=1}^{\ell-2}\sum\limits_{i=j+1}^{\ell -1}j\cdot i.\\
(iii)\ M^{\C^-}_3(G^{\C}_1) & = & \ell-1 +\sum\limits_{i=2}^{\ell}\theta(c_i)\cdot (i-1)+\sum\limits_{j=2}^{\ell-2}\sum\limits_{i= j+1}^{\ell-1}i\cdot (\ell+1- 2i);\\
M^{\C^+}_3(G^{\C}_1) & = & (\ell -1)^2 +\sum\limits_{i=2}^{\ell}\theta(c_i)\cdot (\ell-(i-1))+\sum\limits_{j=2}^{\ell-2}\sum\limits_{i=j+1}^{\ell-1}i\cdot (\ell+1- 2i).
\end{eqnarray*}
\end{cor}
\begin{proof}
For parts (i), (ii) and (iii), the additional sums terms follow from the known results for the corresponding complete graph on vertices $v_{j,1}$, $2\leq j \leq \ell-1$. 
\end{proof}

\subsection{Analysis of chromatic Zagreb indices for $G^{\C}_2$}

\begin{prop}
For the colour cluster $\C$ the graph $G_2$ obtained from the Type-II construction has
\begin{eqnarray*}
(i)\ M^{\C^-}_1(G_2) & = & \sum\limits_{i=1}^{\ell}\theta(c_i)\cdot i^2;\\
M^{\C^+}_1(G_2) & = & \sum\limits_{i=1}^{\ell}\theta(c_i)\cdot (\ell -(i-1))^2.\\
(ii)\ M^{\C^-}_2(G_2) & = & 2\cdot (\theta(c_1)-1) + \sum\limits_{i=1}^{\ell-1}i(i+1)\cdot \theta(c_{i+1});\\
M^{\C^+}_2(G_2) & = & \ell(\ell-1)\cdot (\theta(c_1)-1) + \sum\limits_{i=1}^{\ell-1}i(i+1)\cdot \theta(c_{(\ell-(i-1))}).\\
(iii)\ M^{\C^-}_3(G_2) & = & \sum\limits_{i=1}^{\ell}\theta(c_i) -\ell;\\
M^{\C^+}_3(G_2) & = & (\theta(c_1)-1)^2 +\sum\limits_{i=2}^{\ell}(\ell- (i-1))\cdot \theta(c_i).
\end{eqnarray*}
\end{prop}
\begin{proof}
\begin{enumerate}\itemsep0mm
\item[(i)] Follows directly from the definition of the first chromatic Zagreb index.

\item[(ii)] The first term follows from the edges $v_{2,1}v_{1,i}$, $2\leq i \leq r_1$. The second term follows directly from the definition of the second chromatic Zagreb index.

\item[(iii)]  Define the map $c_1\mapsto c_\ell, c_2\mapsto c_1, c_3\mapsto c_{\ell-1}, c_4\mapsto c_2, c_5 \mapsto c_{\ell-2}$ and so on. Clearly, this mapping ensures maximal adjacency differences for all combinations of the distinct colours. Therefore the result follows through similar reasoning used in part (ii).
\end{enumerate}
\end{proof}

\begin{cor}
For the colour cluster $\C$, the graph $G^{\C}_2$ has
\begin{eqnarray*}
(i)\ M^{\C^-}_1(G^{\C}_2) & = & \sum\limits_{i=1}^{\ell}\theta(c_i)\cdot i^2;\\
M^{\C^+}_1(G^{\C}_2) & = & \sum\limits_{i=1}^{\ell}\theta(c_i)\cdot (\ell -(i-1))^2.\\
(ii)\ M^{\C^-}_2(G^{\C}_2) & = & 2\cdot (\theta(c_1)-1) + \sum\limits_{i=1}^{\ell-1}i(i+1)\cdot \theta(c_{i+1}) +\sum\limits_{j=1}^{\ell-2}\sum\limits_{i=j+2}^{\ell}j\cdot i;\\
M^{\C^+}_2(G^{\C}_2) & = & \ell(\ell-1)\cdot (\theta(c_1)-1) + \sum\limits_{i=1}^{\ell-1}i(i+1)\cdot \theta(c_{(\ell-(i-1))})+\sum\limits_{j=1}^{\ell-2}\sum\limits_{i=j+2}^{\ell}j\cdot i.\\
(iii)\ M^{\C^-}_3(G^{\C}_2) & = & \sum\limits_{i=1}^{\ell}\theta(c_i) -\ell + \sum\limits_{j=1}^{\ell-1}\sum\limits_{i=1}^{j}i - (\ell-2);\\
M^{\C^+}_3(G^{\C}_2) & = & (\theta(c_1)-1)^2 +\sum\limits_{i=2}^{\ell}(\ell- (i-1))\cdot \theta(c_i) + \sum\limits_{j=1}^{\ell-2}\sum\limits_{i=i}^{j}i.
\end{eqnarray*}
\end{cor}
\begin{proof}
For parts (i), (ii) and (iii), the additional sums terms follow from the additional edges added, on completing the graph on the path $v_{1,i}$, $1\leq i \leq \ell$. Note that for $M^{\C^-}_3(G^{\C}_2)$ the initial expression obtain is $M^{\C^-}_3(G^{\C}_2) = \sum\limits_{i=1}^{\ell}\theta(c_i) -\ell +\sum\limits_{j=1}^{\ell-2}~(\sum\limits_{i=j+2}^{(i+(\ell-2))\leq \ell}|i-j|)$ which then simplifies to the result. 
\end{proof}

\section{Application to Certain Positive Non-Increasing Finite Integer Sequences}

In this section, we consider two positive non-increasing finite integer sequences. It lays the basis for investigating other such sequences.

\subsection{Positive integers}

The first sequence is defined as a \textit{mirror image} of the positive integers. Let $s_1=\{a_n\}$, $a_i =\ell -(i-1)$, $1\leq i\leq \ell$. Also, let $\theta(c_i) = a_i$. Consider the corresponding colouring cluster $\C = (\mathcal{C}_i),\ 1\leq i \leq \ell$.

\begin{thm}
Let $\C = (\mathcal{C}_i)$, $1\leq i \leq \ell$, and $\theta(c_i) = a_i$, $a_i \in s_1$. Then, for the colour cluster $\C$ the graph $G^{\C}_1$ obtained from the Type-I construction has
\begin{eqnarray*}
(i)\ M^{\C^-}_1(G^{\C}_1)_{s_1} & = & \sum\limits_{i=1}^{\ell}(\ell -(i-1))\cdot i^2;\\
M^{\C^+}_1(G^{\C}_1)_{s_1} & = &\sum\limits_{i=1}^{\ell}(\ell -(i-1))^3.\\
(ii)\ M^{\C^-}_2(G^{\C}_1)_{s_1} & = & 2(\ell-1) + \sum\limits_{i=2}^{\ell}i\cdot (\ell-(i-1))+\sum\limits_{j=2}^{\ell-1}\sum\limits_{i=j+1}^{\ell}j\cdot i;\\
M^{\C^+}_2(G^{\C}_1)_{s_1} & = & \ell(\ell-1)^2 +\sum\limits_{i=2}^{\ell}\ell \cdot (\ell-(i-1))^2 + \sum\limits_{j=1}^{\ell-2}\sum\limits_{i=j+1}^{\ell-1}j\cdot i.\\
(iii)\ M^{\C^-}_3(G^{\C}_1)_{s_1} & = & \ell-1 +\sum\limits_{i=2}^{\ell} (i-1)\cdot (\ell-(i-1)) + \sum\limits_{j=2}^{\ell-2}\sum\limits_{i=j+1}^{\ell-1}i\cdot (\ell+1-2i);\\
M^{\C^+}_3(G^{\C}_1)_{s_1} & = & (\ell -1)^2 +\sum\limits_{i=2}^{\ell}(\ell-(i-1))^2 + \sum\limits_{j=2}^{\ell-2}\sum\limits_{i=j+1}^{\ell-1}i\cdot (\ell+1-2i).
\end{eqnarray*}

and for the colour cluster $\C$ the graph $G^{\C}_2$ obtained from the Type-II construction has

\begin{eqnarray*}
(i)\ M^{\C^-}_1(G^{\C}_2)_{s_1} & = & \sum\limits_{i=1}^{\ell}(\ell -(i-1))\cdot i^2;\\
M^{\C^+}_1(G^{\C}_2)_{s_1} & = & \sum\limits_{i=1}^{\ell}(\ell -(i-1)^3.\\
(ii)\ M^{\C^-}_2(G^{\C}_2)_{s_1} & = & 2\cdot (\ell-1) + \sum\limits_{i=1}^{\ell-1}i(i+1)\cdot (\ell-i)+ \sum\limits_{j=1}^{\ell-2}\sum\limits_{i=j+2}^{\ell}j\cdot i;\\
M^{\C^+}_2(G^{\C}_2)_{s_1} & = & \ell(\ell-1)\cdot (\ell-1) + \sum\limits_{i=1}^{\ell-1}2i(i+1)(i-1) + \sum\limits_{j=1}^{\ell-2}\sum\limits_{i=j+2}^{\ell}j\cdot i.\\
(iii) M^{\C^-}_3(G^{\C}_2)_{s_1} & = & \sum\limits_{i=1}^{\ell}(\ell-(i-1)) -\ell + \sum\limits_{j=1}^{\ell-1}\sum\limits_{i=1}^{j}i - (\ell-2);\\
M^{\C^+}_3(G^{\C}_2)_{s_1} & = & (\ell-1)^2 +\sum\limits_{i=2}^{\ell}(\ell- (i-1))^2 + \sum\limits_{j=1}^{\ell-2}\sum\limits_{i=i}^{j}i.
\end{eqnarray*}
\end{thm}
\begin{proof}
For Type-I graphical embodiment, we have

\vspace{0.2cm}

\ni Part (i)(a): Follows from definition of the first chromatic Zagreb index, the Type-I graphical embodiment and the fact that $\theta(c_i) = \ell-(i-1)$. 

\vspace{0.2cm}

\ni Part (i)(b):  Follows from definition of the first chromatic Zagreb index, the Type-I graphical embodiment and the fact that $\theta(c_i) = \ell-(i-1)$ and the mapping $c_i\mapsto c_{\ell-(i-1)}$.

\vspace{0.2cm}

\ni Part (ii)(a):  Follows from definition of the second chromatic Zagreb index, the Type-I graphical embodiment and the fact that $\theta(c_i) = \ell-(i-1)$.

\vspace{0.2cm}

\ni Part (ii)(b):  Follows from definition of the second chromatic Zagreb index, the Type-I graphical embodiment and the fact that $\theta(c_i) = \ell-(i-1)$ and the mapping $c_i\mapsto c_{\ell-(i-1)}$.

\vspace{0.2cm}

\ni Part (iii)(a):  Follows from definition of the third chromatic Zagreb index, the Type-I graphical embodiment and the fact that $\theta(c_i) = \ell-(i-1)$.

\vspace{0.2cm}

\ni Part (iii)(b):  Follows from definition of the third chromatic Zagreb index, the Type-I graphical embodiment and the fact that $\theta(c_i) = \ell-(i-1)$ and the mapping $c_\ell\mapsto c_1, c_1\mapsto c_2,c_2\mapsto c_3 \ldots c_{\ell-2}\mapsto c_{\ell-1}$.

\vspace{0.25cm}

\ni and for Type-II graphical embodiment, we have

\vspace{0.2cm}

\ni Part (i)(a): Follows from definition of the first chromatic Zagreb index, the Type-II graphical embodiment and the fact that $\theta(c_i) = \ell-(i-1)$.

\vspace{0.2cm}

\ni Part (i)(b):  Follows from definition of the first chromatic Zagreb index, the Type-II graphical embodiment and the fact that $\theta(c_i) = \ell-(i-1)$ and the mapping $c_i\mapsto c_{\ell-(i-1)}$.

\vspace{0.2cm}

\ni Part (ii)(a):  Follows from definition of the second chromatic Zagreb index, the Type-II graphical embodiment and the fact that $\theta(c_i) = \ell-(i-1)$.

\vspace{0.2cm}

\ni Part (ii)(b):  Follows from definition of the second chromatic Zagreb index, the Type-II graphical embodiment and the fact that $\theta(c_i) = \ell-(i-1)$ and the mapping $c_i\mapsto c_{\ell-(i-1)}$. Also note that $\sum\limits_{j=1}^{\ell-2}\sum\limits_{i=j+2}^{\ell}j\cdot i = (\sum\limits_{i=1}^{\ell-1}(i(\ell-i) - \sum\limits_{i=1}^{\ell-2}i)$.

\vspace{0.2cm}

\ni Part (iii)(a):  Follows from definition of the third chromatic Zagreb index, the Type- II graphical embodiment and the fact that $\theta(c_i) = \ell-(i-1)$.

\vspace{0.2cm}

\ni Part (iii)(b):  Follows from definition of the third chromatic Zagreb index, the Type-II graphical embodiment and the fact that $\theta(c_i) = \ell-(i-1)$ and the mapping $c_\ell\mapsto c_1, c_1\mapsto c_2,c_2\mapsto c_3 \ldots c_{\ell-2}\mapsto c_{\ell-1}$.
\end{proof}

\ni 
\begin{rem}{\rm 
The summation identities $\sum\limits_{i=1}^{\ell}i = \frac{n(n+1)}{2}$, $\sum\limits_{i=1}^{\ell}i^2 =\frac{n(n+1)(2n+1)}{6}$, $\sum\limits_{i=1}^{\ell}i^3 =(\frac{n(n+1)}{2})^2$, $\sum\limits_{i=1}^{\ell}i^4 =\frac{6\ell^5 +15\ell^4+10\ell^3-\ell}{30}$ find partial closure of the results above. For example, $M^{\C^-}_1(G^{\C}_1)_{s_1} = \sum\limits_{i=1}^{\ell}(\ell -(i-1))\cdot i^2 = \frac{\ell^4 +4\ell^3 + 5\ell^2 +2\ell}{12}$. The partial closure of the other results can also be established in a straight forward manner.
}\end{rem}

\subsection{Fibonacci numbers}

The second sequence we consider is the \textit{mirror image} of the Fibonacci sequence. Let $f_0=0$, $f_1=1$ and $f_n = f_{n-1} + f_{n-2}$, where $n \geq 2$. Let $s_2 = \{a_n\}$, $a_i=f_{\ell -(i-1)}$, $1\leq i \leq \ell$.

\begin{thm}
Let $\C = (\mathcal{C}_i)$, $1\leq i \leq \ell$, and $\theta(c_i) = a_i$, $a_i \in s_2$. Then, for the colour cluster $\C$ the graph $G^{\C}_1$ obtained from the Type-I construction has

\begin{eqnarray*}
(i)\ M^{\C^-}_1(G^{\C}_1)_{s_2} & = & \sum\limits_{i=1}^{\ell}f_{\ell -(i-1)}\cdot i^2;\\
M^{\C^+}_1(G^{\C}_1)_{s_2} & = & \sum\limits_{i=1}^{\ell}f_{\ell -(i-1)}\cdot(\ell -(i-1))^2.\\
(ii)\ M^{\C^-}_2(G^{\C}_1)_{s_2} & = & 2(f_\ell-1) + \sum\limits_{i=2}^{\ell}i\cdot f_{(\ell-(i-1))}+\sum\limits_{j=2}^{\ell-1}\sum\limits_{i=j+1}^{\ell}j\cdot i;\\
M^{\C^+}_2(G^{\C}_1)_{s_2} & = & \ell\cdot (f_\ell-1)^2 +\sum\limits_{i=2}^{\ell}f_{\ell-(i-1)}\cdot \ell (\ell-(i-1)) + \sum\limits_{j=1}^{\ell-2}\sum\limits_{i=j+1}^{\ell-1}j\cdot i.\\
(iii)\ M^{\C^-}_3(G^{\C}_1)_{s_2} & = & \ell-1 +\sum\limits_{i=2}^{\ell}f_{\ell-(i-1)}\cdot (i-1) + \sum\limits_{j=2}^{\ell-2}\sum\limits_{i=j+1}^{\ell-1}i\cdot (\ell+1-2i);\\
M^{\C^+}_3(G^{\C}_1)_{s_2} & = & (\ell -1)^2 +\sum\limits_{i=2}^{\ell}f_{\ell-(i-1)}\cdot (\ell-(i-1)) + \sum\limits_{j=2}^{\ell-2}\sum\limits_{i=j+1}^{\ell-1}i\cdot (\ell+1-2i).
\end{eqnarray*}

and for the colour cluster $\C$ the graph $G^{\C}_2$ obtained from the Type-II construction has
\begin{eqnarray*}
(i)\ M^{\C^-}_1(G^{\C}_2)_{s_2} & = & \sum\limits_{i=1}^{\ell}f_{\ell -(i-1)}\cdot i^2;\\
M^{\C^+}_1(G^{\C}_2)_{s_2} & = & \sum\limits_{i=1}^{\ell}f_{\ell -(i-1)}\cdot (\ell -(i-1))^2.\\
(ii)\ M^{\C^-}_2(G^{\C}_2)_{s_2} & = & 2\cdot (f_\ell-1) + \sum\limits_{i=1}^{\ell-1}f_{i-1}\cdot i(i+1) + \sum\limits_{j=1}^{\ell-2}\sum\limits_{i=j+2}^{\ell}j\cdot i;\\
M^{\C^+}_2(G^{\C}_2)_{s_2} & = & \ell(\ell-1)\cdot (f_\ell-1) + \sum\limits_{i=1}^{\ell-1}f_i\cdot i(i+1) + \sum\limits_{j=1}^{\ell-2}\sum\limits_{i=j+2}^{\ell}j\cdot i.\\
(iii)\ M^{\C^-}_3(G^{\C}_2)_{s_2} & = & \sum\limits_{i=1}^{\ell}f_{\ell-(i-1)} -\ell + \sum\limits_{j=1}^{\ell-1}\sum\limits_{i=1}^{j}i - (\ell-2);\\
M^{\C^+}_3(G^{\C}_2)_{s_2} & = & (f_\ell-1)^2 +\sum\limits_{i=2}^{\ell}f_{\ell-(i-1)}\cdot (\ell- (i-1)) + \sum\limits_{j=1}^{\ell-2}\sum\limits_{i=i}^{j}i.
\end{eqnarray*}
\end{thm}
\begin{proof}
For Type-I graphical embodiment, we have

\vspace{0.2cm}

\ni Part (i)(a): Follows from definition of the first chromatic Zagreb index, the Type-I graphical embodiment and the fact that $\theta(c_i) = f_{\ell-(i-1)}$.

\vspace{0.2cm}

\ni Part (i)(b):  Follows from definition of the first chromatic Zagreb index, the Type-I graphical embodiment and the fact that $\theta(c_i) = f_{\ell-(i-1)}$ and the mapping $c_i\mapsto c_{\ell-(i-1)}$.

\vspace{0.2cm}

\ni Part (ii)(a):  Follows from definition of the second chromatic Zagreb index, the Type-I graphical embodiment and the fact that $\theta(c_i) = f_{\ell-(i-1)}$.

\vspace{0.2cm}

\ni Part (ii)(b):  Follows from definition of the second chromatic Zagreb index, the Type-I graphical embodiment and the fact that $\theta(c_i) = f_{\ell-(i-1)}$ and the mapping $c_i\mapsto c_{\ell-(i-1)}$.

\vspace{0.2cm}

\ni Part (iii)(a):  Follows from definition of the third chromatic Zagreb index, the Type-I graphical embodiment and the fact that $\theta(c_i) = f_{\ell-(i-1)}$.

\vspace{0.2cm}

\ni Part (iii)(b):  Follows from definition of the third chromatic Zagreb index, the Type-I graphical embodiment and the fact that $\theta(c_i) = f_{\ell-(i-1)}$ and the mapping $c_\ell\mapsto c_1, c_1\mapsto c_2,c_2\mapsto c_3 \ldots c_{\ell-2}\mapsto c_{\ell-1}$.

\vspace{0.25cm}

\ni For Type-II graphical embodiment, we have

\vspace{0.2cm}

\ni Part (i)(a): Follows from definition of the first chromatic Zagreb index, the Type-II graphical embodiment and the fact that $\theta(c_i) = f_{\ell-(i-1)}$.

\vspace{0.2cm}

\ni Part (i)(b):  Follows from definition of the first chromatic Zagreb index, the Type-II graphical embodiment and the fact that $\theta(c_i) = f_{\ell-(i-1)}$ and the mapping $c_i\mapsto c_{\ell-(i-1)}$.

\vspace{0.2cm}

\ni Part (ii)(a):  Follows from definition of the second chromatic Zagreb index, the Type-II graphical embodiment and the fact that $\theta(c_i) = f_{\ell-(i-1)}$.

\vspace{0.2cm}

\ni Part (ii)(b):  Follows from definition of the second chromatic Zagreb index, the Type-II graphical embodiment and the fact that $\theta(c_i) = f_{\ell-(i-1)}$ and the mapping $c_i\mapsto c_{\ell-(i-1)}$. Also note that $\sum\limits_{j=1}^{\ell-2}\sum\limits_{i=j+2}^{\ell}j\cdot i = (\sum\limits_{i=1}^{\ell-1}(i(\ell-i) - \sum\limits_{i=1}^{\ell-2}i)$.

\vspace{0.2cm}

\ni Part (iii)(a):  Follows from definition of the third chromatic Zagreb index, the Type-I graphical embodiment and the fact that $\theta(c_i) = f_{\ell-(i-1)}$.

\vspace{0.2cm}

\ni Part (iii)(b):  Follows from definition of the third chromatic Zagreb index, the Type-II graphical embodiment and the fact that $\theta(c_i) = f_{\ell-(i-1)}$ and the mapping $c_\ell\mapsto c_1, c_1\mapsto c_2,c_2\mapsto c_3 \ldots c_{\ell-2}\mapsto c_{\ell-1}$.
\end{proof}

\begin{rem}{\rm 
The Fibonacci summation identities $\sum\limits_{i=a}^{\ell}f_i = f_{\ell+2} - f_{a+1}$,  $\sum\limits_{i=1}^{n}f^2_i =f_{n+1}\cdot f_{n+2}$, $\sum\limits_{i=1}^{\ell}f_i^3 = f_{3\ell+2} + 6(-1)^{\ell+1}\cdot f_{\ell-1} + 5$ and $\sum\limits_{i=1}^{\ell}f_i^4 =f_{4\ell+2} + 4(-1)^{\ell+1}\cdot f_{2\ell+1} + 6\ell + 3$ find partial closure of the results above. These exercises are left for the reader. For useful references see \cite{VCH,KSR, DZ1}.
}\end{rem}

\section{Conclusion}

Note that any tree $T$ that allows $\C$ as a proper colouring is a construction type, say Type-$i$. It then follows that any connected subtree $T' \subseteq T$ of order $\ell$ that has exactly one coloured vertex of each distinct colour in $\C$, can be completed by adding exactly $\frac{1}{2}(\ell-1)(\ell-2)$ edges to obtain $G^{\C}$. It is important to note that not all constructions types have this property and therefore, cannot yield the graphical embodiment $G^{\C}$.

\begin{ill}{\rm 
Consider the colour cluster $\C = ((c_1,c_1,c_1,c_1,c_1,c_1),(c_2,c_2,c_2,c_2),\\ (c_3,c_3,c_3),(c_4,c_4))$. The path $v_{1,1}v_{2,1}v_{1,2}v_{2,2}v_{1,3}v_{2,3}v_{1,4}v_{2,4}v_{1,5}v_{3,1}v_{1,6}v_{3,2}v_{4,1}v_{3,3}v_{4,2}$ represents a path-type construction to obtain $G_p$ with minimum edges allowing $\C$ as a proper colouring. However, a path subgraph that contains colours $c_1,c_2,c_3$ cannot also contain colour $c_4$. Similarly, a path subgraph that contains colours $c_1,c_3,c_4$ cannot also contain colour $c_2$. Finally, path subgraphs that contain either $c_2,c_3,c_4$ or $c_1,c_2,c_4$ do not exist. Therefore,  we need more than $\frac{1}{2}(\ell-1)(\ell -2)$ additional edges in order to construct $G^{\C}_p$.
}\end{ill}

The definitions of minimum and maximum chromatic Zagreb indices are reliant on the premises of a given (fixed) graph construction. This study generalised to two primary variants i.e. graph construction and colour cluster for a given number of distinct colours. Note that the respective $|\mathcal{C}_i|= r_i=\theta(c_i) \geq 1$, $\forall$ $i\in \N$.

Hence, although this study considered the Type-I and II construction followed by the analysis of chromatic Zagreb indices of both the positive integers and the Fibonacci sequence, the construction type which results in the $min\{min\}$, $max\{min\}$ and the $min\{max\}$, $max\{max\}$ values for the chromatic Zagreb indices $M^{\C^-}_i$, $M^{\C^+}_i,\ i = 1,2,3$ respectively, remains elusive at this stage.
 
Generally, the results are computationally complex and it is indicative that deeper results will most likely require advanced computer analysis. The application to colour clusters for which the colour weights represent other defined non-negative integer sequences or specialised number sequences opens a wide scope for further research.

Finally it is suggested that the Type-II graphical embodiment can be described recursively as the colour cluster extends from $\C = (\mathcal{C}_i)$, $1\leq i \leq \ell$ to $\C' = (\mathcal{C}_i)$, $1\leq i \leq \ell+1$.

\end{document}